\documentclass[onefignum,onetabnum]{siamonline190516}

\usepackage{amssymb}
\usepackage{amscd}
\usepackage{mathtools}
\usepackage{amsfonts}
\usepackage[english]{babel}

\usepackage{pgfplots}

\newsiamthm{claim}{Claim}
\newsiamremark{remark}{Remark}
\newsiamremark{hypothesis}{Hypothesis}
\crefname{hypothesis}{Hypothesis}{Hypotheses}

\usepackage{algorithmic}

\usepackage{graphicx,epstopdf}

\Crefname{ALC@unique}{Line}{Lines}

\usepackage{amsopn}

\newcommand{\RR}{\mathbb{R}}

\newcommand{\ep}{\varepsilon}

\headers{Multispecies non-local advection models}{V. Giunta, T. Hillen, M. A. Lewis, J. R. Potts}

\title{Local and Global Existence for Non-local Multi-Species\\ Advection-Diffusion Models
	\thanks{Submitted to the editors \textcolor{red}{DATE}.
		\funding{VG and JRP acknowledge the support of an Engineering and Physical Sciences Research Council (EPSRC) grant EP/V002988/1. VG also acknowledges support from GNFM-INDAM. TH is grateful to support from the Natural Science and Engineering Council of Canada Discovery Grant RGPIN-2017-04158. MAL gratefully acknowledges support from the NSERC Discovery and Canada Research Chair programs.}}}

\author{Valeria Giunta\thanks{School of Mathematics and Statistics, University of Sheffield, Hicks Building, Hounsfield Road, Sheffield, S3 7RH, UK
		(\email{v.giunta@sheffield.ac.uk}, \email{j.potts@sheffield.ac.uk}).}
	\and Thomas Hillen \thanks{Department of Mathematical and Statistical Sciences, University of Alberta, Edmonton T6G 2G1, Alberta, Canada
		(\email{thillen@ualberta.ca}).} \and Mark A. Lewis \thanks{Department of Mathematical and Statistical Sciences and Department of Biological Sciences, University of Alberta, Edmonton T6G 2G1, Alberta, Canada (\email{mark.lewis@ualberta.ca})}
	\and Jonathan R. Potts\footnotemark[2]}

\usepackage{amsopn}

\begin{document}
	
	\maketitle

	\begin{abstract}
		Non-local advection is a key process in a range of biological systems, from cells within individuals to the movement of whole organisms.  Consequently, in recent years, there has been increasing attention on modelling non-local advection mathematically.  These often take the form of partial differential equations, with integral terms modelling the non-locality.  One common formalism is the aggregation-diffusion equation, a class of advection diffusion models with non-local advection.  This was originally used to model a single population, but has recently been extended to the multi-species case to model the way organisms may alter their movement in the presence of coexistent species.  Here we prove existence theorems for a class of non-local multi-species advection-diffusion models, with an arbitrary number of co-existent species.  We prove global existence for models in $n=1$ spatial dimension and local existence for $n>1$.  We describe an efficient spectral method for numerically solving these models and provide example simulation output.  Overall, this helps provide a solid mathematical foundation for studying the effect of inter-species interactions on movement and space use.	
	\end{abstract}



	\section{Introduction}

	It is essential for individuals, whether cells or animals, to gain information about their local environment \cite{turchin1989,RaiEcology}. Not only do individuals sense environmental features, such as food, temperature, pH-level, and so on, they also are able to detect other individuals in a local spatial neighborhood, such as predators, prey, or conspecifics \cite{CosnerCantrellbook,okubolevin2013}. This feature is not only restricted to higher level species, but is also found in cells \cite{HilLewis2014}. For example human immune cells gather information about their tissue environment and they are able to distinguish friend from foe \cite{shahriyari16,eftimie11}. The process of gaining information about presence or absence of other species in the environment is intrinsically non-local \cite{Eftimienonlocal,Lutscherbook}. 	Mathematically, the non-local sensing of neighboring individuals leads to non-local advection terms in the corresponding continuum models, and that is the topic of this paper.

	Non-local advection is a mechanism underlying a wide range of biological systems.  In ecology, animals sense their surroundings and make decisions to avoid predators, find prey, and/or aggregate in swarms, flocks or herds \cite{Eftimienonlocal,CuckerSmale,Eftimie2007,Levine00,mogilner1999non}. This non-local sensing can occur on several scales, from near to far \cite{benhamou2014,bastille2018spatial,martinez2020range}.  These scales affect the overall spatial arrangement of populations \cite{potts2014predicting,buttenschon2020non,armstrong2006continuum} and can lead to species aggregation, segregation, and also more complex mixing patterns \cite{Eftimie2007,Eftimie2,pottslewis2019}.  Whereas animals can sense and interact over distances using sight, smell and hearing, in cell biology, cells interact non-locally by extending long thin protrusions, probing the environment  \cite{armstrong2006continuum, painteretal2015,Osswald}. % https://www.ncbi.nlm.nih.gov/pmc/articles/PMC4226656/
	Chemotaxis processes, leading to the following of chemical trails by organisms, can also be formulated as non-local advective processes \cite{hilpainterschmeiser,shi2021spatial}, and have been observed in taxa from single-celled organisms to insect populations to large vertebrate animals \cite{hillen2009user}. %The resulting equations may be non-local in time as well as space.
	
	% Importance of non-locality for well-posedness of mathematical models, but technical difficulties
	From a mathematical modelling perspective, non-locality in continuum models often arises as an integral term inside a derivative. The corresponding models become intrinsically non-local, and classical theories, developed for local models, no longer apply \cite{buttenschon2020non,carrillo2019aggregation}. Non-local terms in continuum models offer new challenges and new opportunitites \cite{bedrossian2011local,bertozzi2011lp,pottslewis2016b,mogilner1999non,buttenschon2020non}.   For example, in single-species models of aggregation, the structure of the non-local advective term is fundamental for avoiding blow-up and ensuring global existence of solutions \cite{hilpainterschmeiser,dolbeault2004optimal, bertozzi2007finite, carrillo2019aggregation}.  In models of home ranges \cite{briscoeetal2002} and territory formation \cite{pottslewis2016a}, non-local advection is necessary for ensuring well-posedness.  In the context of modelling swarm dynamics, \cite{mogilner1999non} showed that non-local advection is vital for the formation of cohesive swarms.
	
	%\color{black}
	
	% Aggregation-diffusion equations as a prime example, and generalisations to >1 species
	Consequently, non-local advection has become a popular feature of biological models \cite{buttenschon2020non}.  One common class of such models is the aggregation-diffusion equation \cite{topaz2006nonlocal,delgadino2019uniqueness}.  This models a single population, $u(x,t)$, that undergoes diffusion and non-local self-attractive advection, leading to the following general form \cite{carrillo2019aggregation}
	\begin{equation}\label{eq:aggeq}
	%\frac{\partial u}{\partial t}=\nabla \cdot[u\nabla(F(u) - K\ast u)],
	\frac{\partial u}{\partial t}=\Delta u^m- \nabla \cdot[u \nabla(K\ast u)],
	\end{equation}
	where %$F(u)$ determines the type of diffusion (e.g. $F(u)=D \ln(u)$ is linear diffusion where $D$ is a constant) and 
	$K \ast u$ is the convolution of $u$ with a spatial averaging kernel, $K$, and $m$ is a positive integer.   As such, the structure of $K$ models the non-local interactions of the population with itself.  
	%Often, it is assumed that $\nabla \cdot[u\nabla(F(u)]=\Delta u^m$ in the definition of aggregation-diffusion equation (so that $F(u)=\ln(u)$ for $m=1$ or $F(u)=mu^{m}/(m-1)$ for $m>1$) 
	\Cref{eq:aggeq} can lead to the spontaneous formation of non-uniform patterns, consisting of single or multiple stationary aggregations of various shapes and sizes, under certain conditions \cite{james2015numerical, craig2016blob}.  However, there is numerical evidence that the multiple-aggregation case is often, and possibly always, metastable \cite{topaz2006nonlocal, burger2014stationary, carrillo2019aggregation}.
	
	One can readily generalise the aggregation-diffusion equation to the multi-species situation as follows: 
	\begin{align}
	%\frac{\partial u_i}{\partial t}=\nabla \cdot\left[u_i\nabla\left(F_i(u_i) - \sum_{j=1}^N h_{ij}K\ast u_j\right)\right],
	\frac{\partial u_i}{\partial t}=D_i\Delta u_i^m-\nabla \cdot\left[u_i\nabla\sum_{j=1}^N h_{ij}K\ast u_j\right],
	\label{eq:msagg}
	\end{align}
	where $u_1(x,t),\dots,u_N(x,t)$ are locational densities of $N\geq 1$  populations at time $t$, $D_i \in {\mathbb R}_{>0}$ is the diffusion constant of population $i$, and  $h_{ij} \in {\mathbb R}$ are constants denoting the attractive (if $h_{ij}>0$) or repulsive (if $h_{ij}<0$) tendencies of population $i$ to population $j$.  
	Indeed, the $N=2$ case has received some attention \cite{evers2017equilibria, carrillo2018zoology}, with equations of the same or similar form to \Cref{eq:msagg} being applied to predator-prey dynamics \cite{fagioli2019multiple}, animal territoriality \cite{pottslewis2016a}, cell-sorting \cite{painteretal2015} as well as human gangs \cite{barbaro2020analysis}.  For $N=2$, it is possible to observe both aggregation and segregation patterns emerge, depending on the relative values of the $h_{ij}$ constants \cite{evers2017equilibria, pottslewis2019}.
	
	% The Potts-Lewis model, its complex pattern formation properties, and relevance to ecosystem modelling
	An example of \Cref{eq:msagg} where $N$ is arbitrary was proposed by \cite{pottslewis2019} as a model of animal ecosystems.   The authors assumed that each population can detect the population density of other populations over a local spatial neighbourhood.  The mechanism behind this detection could have various forms, three of which are explained in \cite{pottslewis2019}: direct observations of individuals at a distance, indirect communication via marking the environment (e.g. using urine or faeces), and memory of past interactions with other populations.  \cite{pottslewis2019} showed that all three of these biological mechanisms lead to the same multi-species aggregation-diffusion model in the appropriate adiabatic limit. The authors analysed pattern formation properties of \Cref{eq:msagg} where the diffusion term is linear, 
	%i.e. $F(u_i)=D_i \ln(u_i)$, 
	i.e. $m=1$, 
	in one spatial dimension with periodic boundary conditions.  They further assumed that $K(x)$ is a top-hat kernel, i.e. $K(x)=1/(2\delta)$ for $x \in (-\delta,\delta)$ and $K(x)=0$ otherwise, and also that $j \neq i$ (i.e. no self-attraction or repulsion).  With these assumptions in place, the authors showed that, whilst the pattern formation properties when $N=2$ can be fully categorised, the $N=3$ case is much richer.  Indeed, numerical analysis for $N=3$ revealed stationary patterns, regular oscillations, period-doubling bifurcations, and irregular spatio-temporal patterns suggestive of chaos \cite{pottslewis2019}.  
	
	These insights highlighted the importance of understanding non-linear, non-local feedbacks between the locations of animal populations.  In the ecological literature, the field of {\em Species Distribution Modelling} (SDM) is dominated by efforts to find correlations between animal locations and environmental features \cite{araujo2006five, zimmermann2010new}.  These features are then used to predict species distributions in either new locations or future environmental conditions \cite{beaumont2008choice, marmion2009evaluation} and hence inform conservation actions \cite{villero2017integrating}.  However, despite considerable research effort into SDMs, a recent meta-analysis of 33 different SDM approaches revealed that none of the models studied were good at making predictions in a range of novel situations \cite{norberg2019comprehensive}. 
	Based on the results of \cite{pottslewis2019}, we conjecture that this may be, in part, due to a failure of these models to account for non-linear feedbacks in movement mechanisms.  We propose that employing a multi-species aggregation-diffusion approach, typified by \Cref{eq:msagg}, may help improve predictive performance when modelling the spatial distributions of animal populations. 
	
	% Outstanding issues in this model and how we deal with them here
	As a step to this end, the aim of this paper is twofold: to begin building solid mathematical foundations underlying the model and observations of \cite{pottslewis2019}, and to construct an efficient numerical scheme for future investigations.  For our mathematical analysis, we are able to drop the assumption from \cite{pottslewis2019} that $j \neq i$, thus allowing for self attraction or repulsion.  However, we have to assume that $K$ is twice differentiable, so cannot be the same top-hat function used by \cite{pottslewis2019} but can be a smooth approximation of the top-hat function.  With these assumptions in place, we prove the global existence of a unique, positive solution in one spatial dimension and local existence (up to a finite time $T_*$) in arbitrary dimensions. We also propose an efficient scheme for solving multi-species aggregation-diffusion models numerically, based on a spectral method, and give some example output of both stationary and fluctuating patterns.
	
	% Paper organisation
	Our paper is organised as follows. \Cref{sec:mod} introduces the study system and states the main results (global existence and positivity in one spatial dimension; local existence in arbitrary dimensions).  In \Cref{sec:analytic} we prove the main results. \Cref{sec:numeric} details a method for numerically solving the study system, together with some example numerical output. \Cref{sec:disc} gives a discussion and concluding remarks.
	
	%%%%%%%%%%%%%%%%%%%%%%
	\section{The Model}
	\label{sec:mod}
	
	%	N: number of species\\
	%	n: space dimension\\
	
	We consider $N$ different populations of moving organisms.  These could either be different species or different groups within a species, such as territorial groupings or herds.  In either case, we use the term {\it population} and write $u_i(x,t)$ to denote  the  density of population $i \in \{1,...,N\}$ at time $t$.  As with \Cref{eq:msagg}, we assume that each population detects the population density of other populations over space, and adjusts its directed motion via advection towards a weighted sum of the spatially averaged population densities.  
	
	%We focus our model on timescales over which births and deaths are negligible, thus assuming that $\int_\Omega  u_i(x,t) dx$ does not vary over time (where $\Omega $ is the study area).  Consequently, the spatio-temporal dynamics are driven by the movements of organisms in response to the presence of other populations.  
	
	Before generalising to arbitrary dimensions, we first define our system in one dimension (1D) as follows
	\begin{eqnarray}
	\frac{\partial u_i}{\partial t} &=& D_i\frac{\partial^2 u_i}{\partial x^2}  - \frac{\partial}{\partial x}\left[u_i \frac{\partial}{\partial x}\left(\sum_{j=1}^N h_{ij}\bar u_j\right)\right], \nonumber \\
	%		H &=& (h_{ij})_{i,j} \nonumber \\
	\label{eq:model1}
	\bar u_j(x) &=& (K \ast u_j)(x) := \int_0^L K(x-y)u_j(y) dy.  
	\end{eqnarray}
	We examine this system on a domain $[0,L]$ with periodic boundary conditions, so that $\Omega = [0,L]/\{0,L\}$ (the topological quotient of $ [0,L]$ by $\{0,L\}$).  Here, $K \geq 0$ is a {\it local averaging kernel} (i.e. a probability density function on $\Omega$ with zero mean), $D_i$ is the diffusion constant of population $i$, and $h_{ij}$ is the strength of attraction (resp. repulsion) of population $i$ to (resp. from) population $j$ if $h_{ij}>0$ (resp. $h_{ij}<0$).  The local averaging kernel, $K$, describes the spatial scale over which organisms scan the environment when deciding to move in response to the presence of other populations.  Here, we will assume $K$ is twice differentiable with $\nabla { K} \in L^\infty(\mathbb{T})$.
	
	Notice that $\int_\Omega  u_i(x,t) dx$ does not vary over time so we define a constant $p_i=\int_\Omega  u_i(x,t) dx$ for each $i$.  Consequently, our model is suitable for modelling systems of animal or cell populations over timescales where births and deaths have a negligible effect on the population size.  For example,  for systems of organisms whose population sizes vary by only small amounts across a season (as is the case for many mammals, birds, and reptiles in summer), this could model dynamics over a single season.
	
	We can use vector notation to write System (\ref{eq:model1}) in a more compact form. Let 
	\[ u= (u_1, \dots u_N)^T,  \qquad D=\mbox{diag}(D_1, \dots, D_N), \qquad H=(h_{ij})_{i,j},\]
	where $(h_{ij})_{i,j}$ denotes the matrix whose $i,j$-th entry is $h_{ij}$.  Then System (\ref{eq:model1}) can be written as 
	\begin{equation}\label{eq:model2}
	u_t = Du_{xx} - (u\cdot (H\bar u)_x)_x.
	\end{equation}  
	In higher dimensions we make the analogous assumption that $\Omega \subset \RR^n$ is a periodic domain, i.e. a torus $\mathbb{T}$. Then the system on the general $n$-dimensional torus $\mathbb{T}$ becomes
	
	\begin{equation}\label{eq:model3}
	u_t = D\Delta u - \nabla\cdot(u\cdot \nabla (H\bar u)).
	\end{equation}  
	To avoid confusion in this vector notation we can write each row as
	\[
	u_{it} = D_i\sum_k \frac{\partial^2}{\partial x_k^2}  u_i - \sum_k\frac{\partial}{\partial x_k} \left(u_i \sum_j \frac{\partial}{\partial x_k}  (h_{i j}\bar u_j)\right),
	\]
	which leads to
	\[
	u_t = D\sum_k \frac{\partial^2}{\partial x_k^2}  u - \sum_k\frac{\partial}{\partial x_k} \left(u \circ \sum_j \frac{\partial}{\partial x_k}  (H_{\cdot j}\bar u_j)\right),
	\]
	where $H_{\cdot j}$ is the $j$-th column of $H$ and $\circ$ is the Hadamard product.  
	%Using repeated indices to indicate summation, this can be simplified to:
	%\[
	%u_t = D\frac{\partial}{\partial x_k} \frac{\partial}{\partial x_k}  u - %\frac{\partial}{\partial x_k} \left(u \circ\frac{\partial}{\partial x_k}  %(H_{\cdot j}\bar u_j)\right).
	%\] 
	%	
	%In this manuscript, we will investigate various functional forms for the local averaging kernel, $K$.  We first prove existence of the system by assuming $K$ is twice differentiable with $\nabla { K} \in L^\infty(\mathbb{T})$, which can be dealt with in full generality.  However, after this we show how the proof can be extended to the following two specific (discontinuous) examples
	%\begin{equation}\label{convolution}
	%	K(x)=\chi_\delta(x):= \frac{1}{|B_\delta(0)|}I_{B_\delta(0)}(x), \qquad   %K(x)=\psi_\delta(x):= \frac{1}{|\delta^n|}I_{[-\delta,\delta]^n}(x).
	%\end{equation}
	%Here, $I_S$ is the indicator function on the set $S$ and $B_\delta(0)$ is the ball in the Euclidean-norm of $\mathbb{R}^n$. The example local averaging kernels in Equation (\ref{convolution}) are $n$-dimensional generalisations of the 1D kernel examined by \cite{pottslewis2019}.
	%\\
	%\\
	We now state our main result, as follows.  
	\begin{theorem} 
		\label{thm:main}
		Assume $u_0\in H^2(\mathbb{T} )^N$. If $n\geq 1$ then there exists a time $T_* \in (0,\infty]$ and a unique solution $u(x,t)$ to \Cref{eq:model3}, valid for $t \in [0,T_*)$, such that
		\[ u\in C^1((0,T_*), L^2(\mathbb{T} ))^N \cap C^0([0,T_*), H^2(\mathbb{T} ))^N.\]
		If $n=1$ and $u_0\in C^2(\mathbb{T} )^N$ such that $u_0(x)>0$ for $x \in \mathbb{T}$, then there is a unique positive solution $u(x,t)$ to \Cref{eq:model3} such that
		\[ u\in C^1((0,\infty), L^2(\mathbb{T} ))^N \cap C^0([0,\infty), C^2(\mathbb{T} ))^N.\]
	\end{theorem}
	The first part of this theorem ($n\geq 1$) will follow from \cref{lem:local} and the second ($n=1$) will be established in \cref{thm:global}.
	
	\subsection{Notation}
	
	We will employ the following notation throughout.  Let $ f:L^p(\Omega) \rightarrow \mathbb{R} $.
	\begin{itemize}
		\item $ \|f\|_{L_p} =( \int_{\Omega} |f|^p)^{1/p} $, where $ 1\leq p < \infty $.
		\item $ \| f\|_{L^{\infty}} = inf \{C \geq 0 : |f(x)| \leq C, a.e. \} $.
	\end{itemize}
	Let $ g=(g_1, g_2, \dots, g_N):(L^p)^N \rightarrow \mathbb{R}$. We will use the following norms
	\begin{itemize}
		\item $ \|g\|_{(L_p)^N} = \sum_{i=1}^{N}\| g_i \|_{L^p} $, where $ 1\leq p < \infty $.
		\item $ \| g\|_{(L^{\infty})^N} = \max_{i=1, 2, \dots, N} \{\|g_i\|_{L^{\infty}}\} $.
	\end{itemize}
	To ease the notation, we will usually omit the index $ N $ and write $  \| g\|_{L^{p}} $ instead of $  \| g\|_{(L^{p})^N} $.
	
	%%%%%%%%%%%%%%%%%
	\section{Model Analysis}
	\label{sec:analytic}
	\subsection{Existence and uniqueness of mild solutions}
	\begin{definition}
		Given $u_0\in (L^2(\mathbb{T} ))^N$ and $T>0$. We say that 
		\[ u(x,t) \in L^\infty((0,T), L^2(\mathbb{T} ))^N\]
		is a  \underline{mild solution} of \Cref{eq:model3} if 
		\begin{equation}\label{eq:mild}
		u = e^{D\Delta t} u_0 -  \int_0^t e^{D\Delta (t-s)} \nabla \cdot (u\cdot \nabla ( H\bar u) ) ds ,
		\end{equation}
		for each $0<t\leq T$, where $e^{D\Delta t} $
		denotes the solution semigroup of the heat equation system $u_t = D \Delta u$ on $\mathbb{T} $, i.e. on $\Omega$  with periodic boundary conditions. 
	\end{definition}  
	The crucial term in (\ref{eq:model2}) is the non-local term $H\bar u$ and the following a-priori estimates for $\bar u$ are essential for the existence theory of this model. We will consider convolution with an appropriately smooth kernel, $K$.  Eventually, in \cref{l:deltainfty}, we will need to assume that $K$ is twice differentiable, but the first two Lemmas only require $K$ to be (once) differentiable, so we state them in this more general case.
	
	\begin{lemma}\label{l2:h1}
		Let $\varphi \in L^2(\mathbb{T})$ and $ K:\mathbb{T}\rightarrow {\mathbb R}$ be differentiable.  Then $\|\bar\varphi\|_{H^1}=\|K\ast\varphi\|_{H^1}\leq(\|K\|_{L^1}+\|\nabla K\|_{L^1})\|\varphi\|_{L^2}$. 
	\end{lemma}
	\begin{proof}
		First, $\|{K}\ast\varphi\|_{H^1}=\|{K}\ast\varphi\|_{L^2}+\|\nabla({K}\ast\varphi)\|_{L^2}$. We also observe that $ \nabla( K \ast \varphi) =  \nabla K \ast \varphi = (\partial_{x_1} K \ast \varphi, \partial_{x_2} K \ast \varphi, \dots, \partial_{x_n} K \ast \varphi) $. Then, applying Young's convolution inequality to both summands, we have $\|{K}\ast \varphi\|_{L^2} \leq \|{K}\|_{L^1}\|\varphi\|_{L^2}$ and $\|\nabla({K}\ast\varphi)\|_{L^2}=\|(\nabla{K})\ast\varphi\|_{L^2} = \|(\partial_{x_1} K \ast \varphi, \partial_{x_2} K \ast \varphi, \dots, \partial_{x_n} K \ast \varphi)\|_{L^2} = \sum_{i=1}^{n}\|\partial_{x_i}{K}\ast \varphi \|_{L^2} \leq  \sum_{i=1}^{n} \| \partial_{x_i} K\|_{L^1} \| \varphi \|_{L^2} = \| \nabla K \|_{L^1}  \| \varphi \|_{L^2} $, proving the lemma.
	\end{proof}

	\begin{lemma}\label{l2:h1b}
		Let $\varphi \in L^\infty(\mathbb{T})$ and ${ K}:\mathbb{T}\rightarrow {\mathbb R}$ be differentiable with $\nabla {K} \in L^\infty(\mathbb{T})$.  Then $\|\nabla {K}\ast\varphi\|_{L^\infty}\leq|\mathbb{T}|^{1/2}\|\nabla {K}\|_{L^\infty}\|\varphi\|_{L^2}$. 
	\end{lemma}
	\begin{proof}
		First note that 
		\begin{align*}
		\|\nabla({ K}\ast\varphi)\|_{L^\infty}&=\|(\nabla{K})\ast\varphi\|_{L^\infty}\\
		&  = \| (\partial_{x_1} K \ast \varphi, \partial_{x_2} K \ast \varphi, \dots, \partial_{x_n} K \ast \varphi)  \|_{L^\infty} \\
		& = \max_{i=1, 2, \dots, n} \{ \| \partial_{x_i} K \ast \varphi\|_{L^\infty}  \} \\
		& \leq \max_{i=1, 2, \dots, n} \{ \|\partial_{x_i} K\|_{L^{\infty}} \|\varphi \|_{L^1} \} \\
		&  =  \max_{i=1, 2, \dots, n} \{ \|\partial_{x_i} K\|_{L^{\infty}}\} \|\varphi \|_{L^1} \\
		&  = \|\nabla{K}\|_{L^\infty}\|\varphi\|_{L^1},
		\end{align*}
		using Young's convolution inequality in the fourth line. Then, since $\mathbb{T}$ is of finite measure in $\mathbb{R}^N$, we have $\|\varphi\|_{L^1} \leq |\mathbb{T}|^{1/2}\|\varphi\|_{L^2}$ (this step uses H\"older's inequality, applied to $\|{\mathbf 1}\varphi\|_{L^1}$ where ${\mathbf 1}:{\mathbb T}\rightarrow {\mathbb R}$ such that ${\mathbf 1}(x)=1$).  Hence $\|\nabla{K}\|_{L^\infty}\|\varphi\|_{L^1} \leq |\mathbb{T}|^{1/2}\|\nabla{K}\|_{L^\infty}\|\varphi\|_{L^2}$, proving the lemma.
	\end{proof}

	\begin{lemma}\label{l:deltainfty}
		Let $\varphi \in H^1(\mathbb{T})$  and ${K}:\mathbb{T}\rightarrow {\mathbb R}$ be twice differentiable with $\nabla { K} \in L^\infty(\mathbb{T})$.  Then $
		\| \Delta (K\ast \varphi) \|_{L^\infty}\leq \|\nabla K \|_{L^{\infty}} \| \nabla \varphi\|_{L^2} |\mathbb{T}|^{1/2}$. 
	\end{lemma}
	\begin{proof} First note that
		\begin{align*}
		\|\Delta ({K}\ast\varphi)\|_{L^\infty} &= \left\| \sum_{i=1}^{n} \partial_{x_i}^2 (K \ast \varphi) \right\|_{L^\infty}%= \left\| \sum_{i=1}^{n} \partial_{x_i} (\partial_{x_i}(K \ast \varphi)) \right\|_{L^\infty}
		\\
		& = \left\| \sum_{i=1}^{n} \partial_{x_i} K \ast \partial_{x_i}\varphi \right\|_{L^\infty} \\
		& \leq \sum_{i=1}^{n} \left\| \partial_{x_i} K \ast \partial_{x_i}\varphi \right\|_{L^\infty} \\
		& \leq  \sum_{i=1}^{n} \left\| \partial_{x_i} K \right\|_{L^\infty}  \|\partial_{x_i}\varphi \|_{L^1}\\
		&\leq\|\nabla{K}\|_{L^\infty}\|\nabla\varphi\|_{L^1},
		\end{align*}
		where the second inequality uses Young's convolution inequality. Then, as in \cref{l2:h1b}, we have $\|\nabla\varphi\|_{L^1} \leq |\mathbb{T}|^{1/2}\|\nabla\varphi\|_{L^2}$.  Hence $\|\nabla{K}\|_{L^\infty}\|\nabla\varphi\|_{L^1} \leq |\mathbb{T}|^{1/2}\|\nabla{K}\|_{L^\infty}\|\nabla\varphi\|_{L^2}$, proving the lemma.
	\end{proof}

	Before we formulate the proof of local and global existence, we recall a regularity result for the heat equation semigroup on a torus as formulated by \cite{taylorIII} p.274: 
	%Reference at the end
	\begin{lemma}\label{lemmaTaylor} For all $p\geq q >0$ and $s\geq r$ we have the embedding 
		\[ e^{\Delta t} : W^{r,q}(\mathbb{T}) \to W^{s,p}(\mathbb{T}), \qquad \mbox{with norm } C t^{-\kappa}, \]
		where $C$ is a constant and 
		\[ \kappa = \frac{n}{2} \left(\frac{1}{q} - \frac{1}{p}\right) + \frac{1}{2} (s-r). \]
	\end{lemma} 
	
	\begin{theorem}\label{thm:local} For each $u_0 \in L^2(\mathbb{T} )^N$ there exists a time $T>0$ and a unique mild solution \eqref{eq:mild} of \Cref{eq:model3} with 
		\[ u\in L^\infty((0,T), L^2(\mathbb{T} ))^N.\]
	\end{theorem}
	\begin{proof}
		The proof uses a Banach fixed-point argument. Let $M:=2\|u_0\|_{L^2} $. We define a map  
		\[ v\mapsto Q v :=  e^{D\Delta t} u_0 - \int_0^t e^{D\Delta (t-s)} \nabla\cdot (v \cdot\nabla(H\bar v)) ds,\]for $v\in L^\infty((0,T), L^2(\mathbb{T} ))^N$.
		\\
		\\
		\underline{{\bf Step 1:} $Q$ maps a ball into itself:} Let $B_M(0)\subset L^2(\mathbb{T} )^N$  be the ball of radius $M$ in $L^2(\mathbb{T} )^N$. Let $v=(v_1,\dots,v_N)\in L^\infty((0,T_{min}), B_M(0))^N $, where $T_{min}$ will be determined later.  Writing $u_0=(u_{10},...,u_{N0})$, for each $T \in (0,T_{min})$ we have 
		\begin{align*}
		\|Qv_i\|_{L^2} &\leq  \|u_{i0}\|_{L^2} +  \left\|\int_0^T e^{D\Delta (T-s)}\nabla\cdot  (v_i \nabla ((H\bar v)_i)) ds\right\|_{L^2}\\
		&\leq \|u_{i0}\|_{L^2} +\int_0^T C (T-s)^{-\frac{1}{2}} \|v_i \nabla ((H\bar v)_i) \|_{L^2} ds \\
		&\leq  \|u_{i0}\|_{L^2} + 2C\sqrt{T} \sup_{0<t\leq T} \|v_i \nabla ((H\bar v)_i)\|_{L^2} .
		\end{align*}
		In the second inequality we used the regularizing property of the heat equation semigroup from $H^{-1}$ to $L^2$ with a norm $C t^{-\frac{1}{2}}$, as in \cref{lemmaTaylor}. Since  $ (H \bar v)_i= \sum_{j=1}^{N} h_{ij} K \ast v_j $, we continue the previous estimate as:
		\begin{align*}
		\|Qv_i\|_{L^2} 
		&\leq  \|u_{i0}\|_{L^2} +  2C\sqrt{T}\sup_{0<t\leq T} \left\| v_i \nabla \left( \sum_{j=1}^{N} h_{ij} K \ast v_j\right) \right\|_{L^2} \\
		&\leq  \|u_{i0}\|_{L^2} +  2C\sqrt{T}\sup_{0<t\leq T}  \sum_{j=1}^{N} |h_{ij}| \left\| v_i \nabla \left( K \ast v_j\right) \right\|_{L^2} \\
		&\leq  \|u_{i0}\|_{L^2} +  2C\sqrt{T}  \sum_{j=1}^{N} |h_{ij}| \sup_{0<t\leq T}  n\|v_i\|_{L^2} \left\| \nabla \left( K \ast v_j\right) \right\|_{L^{\infty}} \\
		&\leq  \|u_{i0}\|_{L^2} +  2C\sqrt{T} \|\nabla K\|_{L^{\infty}} |\mathbb{T}|^{1/2} \sum_{j=1}^{N} |h_{ij}| \sup_{0<t\leq T}  n\|v_i\|_{L^2} \|v_j\|_{L^{2}} 
		\end{align*}
		In the third inequality we used H\"older's inequality, and in the last one we used \cref{l2:h1b}. 
		From the previous estimate, we obtain
		\begin{align*}
		\|Qv\|_{L^2}& = \sum_{i=1}^{N}	\|Qv_i\|_{L^2} \\
		& \leq \sum_{i=1}^{N}  \|u_{i0}\|_{L^2} +  2C\sqrt{T} n|\mathbb{T}|^{1/2} \|\nabla K\|_{L^{\infty}}  \sum_{i,j=1}^{N} |h_{ij}| \sup_{0<t\leq T}  \|v_i\|_{L^2} \|v_j\|_{L^{2}} \\
		& \leq\|u_{0}\|_{L^2} +  2C\sqrt{T}n|\mathbb{T}|^{1/2} \|\nabla K\|_{L^{\infty}}  \|H\|_{\infty} \sup_{0<t\leq T}  \|v \|_{L^2}^2, 
		\end{align*}
		where $ \|H\|_{\infty}=\max_{i,j} |h_{ij}| $.	Notice that $\|u_0\|_{L^2} = \frac{M}{2}$, hence we can always find a time $T_1$  small enough such that 
		\[ \sup_{0<t\leq T_1} \|Qv\|_{L^2} \leq M, \]
		%	Indeed, we need to ensure that 
		%	\begin{equation}\label{Teins}
		%		T\leq T_1:= \frac{\delta}{( C \mu  M)^2 }  
		%	\end{equation}       
		so that $ Q v \in L^\infty((0,T_1), B_M(0))^N. $
		\\
		\\
		\underline{{\bf Step 2:} $Q$ is a contraction for $T$ small enough:} Given $v_{1}=(v_{11},...,v_{1N}),v_{2}=(v_{21},...,v_{2N}) \in L^\infty((0,T_{min}), B_M(0))^N$, we compute for $T \in (0,T_{min})$ the following
		\begin{align*}
		\|Qv_{1i}-Qv_{2i}\|_{L^2} &= \left\|  \int_0^T e^{D\Delta (T-s)} \left[\nabla \cdot (v_{1i} \nabla ((H \bar v_1)_i)) - \nabla\cdot (v_{2i}\nabla((H\bar v_2)_i))\right]  ds  \right\|_{L^2}\\
		&\leq   \left\|\int_0^T e^{D\Delta (T-s)} \nabla\cdot ((v_{1i}-v_{2i}) \nabla ((H\bar v_1)_i) ds \right\|_{L^2} \\
		& \hspace*{1cm}+\left\| \int_0^T e^{D\Delta (T-s)} \nabla\cdot [v_{2i}\nabla(H(\bar v_{1i}-\bar v_{2i}))_i] ds \right\|_{L^2} \\
		&\leq  \int_0^T C(T-s)^{-1/2} \|(v_{1i}-v_{2i})\nabla((H \bar v_1)_i)\|_{L^2}ds \\
		& \hspace*{1cm}+\int_0^T C(T-s)^{-1/2} \|v_{2i} \nabla((H (\bar v_{1}-\bar v_{2}))_i)\|_{L^2}ds \\
		%		&\leq  \int_0^T C(T-s)^{-1/2} |{\mathbb T}|^{-1/2} \|(v_1-v_2)\cdot \nabla(H \bar v_1)\|_{L^1}ds\\
		%		& \hspace*{1cm}+\int_0^T C(T-s)^{-1/2} |{\mathbb T}|^{-1/2} \|v_2\cdot \nabla(H (\bar v_1-\bar v_2))\|_{L^1}ds \\
		&\leq  2C\sqrt{T} \sup_{0<t\leq T}(\|(v_{1i}-v_{2i}) \nabla((H \bar v_{1})_i)\|_{L^2}+ \|v_{2i} \nabla((H (\bar v_1-\bar v_2))_i)\|_{L^2}) 
		\end{align*}
		In the second inequality we used the regularizing property of the heat equation semigroup from $H^{-1}$ to $L^2$ with a norm $C t^{-\frac{1}{2}}$, as in \cref{lemmaTaylor}. Since  $ (H \bar v_1)_i= \sum_{j=1}^{N} h_{ij} K \ast v_{1j} $ and $ (H \bar v_2)_i= \sum_{j=1}^{N} h_{ij} K \ast v_{2j} $ we continue the previous estimate as:
		\begin{align*}	
		\|Qv_{1i}-Qv_{2i}\|_{L^2} 
		&\leq  2C\sqrt{T} \sup_{0<t\leq T}\left(\left\|(v_{1i}-v_{2i}) \sum_{j=1}^{N} |h_{ij}| (\nabla K \ast\ v_{1j})\right\|_{L^2} \right.\\
		&\hspace*{1cm}\left. + \left\|v_{2i}  \sum_{j=1}^{N} |h_{ij}|( \nabla K \ast (v_{1j}- v_{2j}))\right\|_{L^2}\right) \\		
		&\leq  2C\sqrt{T} \sup_{0<t\leq T}(\|v_{1i}-v_{2i}\|_{L^2} n\sum_{j=1}^{N} |h_{ij}| \|\nabla K \ast\ v_{1j}\|_{L^\infty}\\
		&\hspace*{1cm} + \|v_{2i}\|_{L^2}  n\sum_{j=1}^{N} |h_{ij}| \|\nabla K \ast ( v_{1j}- v_{2j})\|_{L^\infty}) \\
		&\leq  2C\sqrt{T}\|H\|_\infty \|\nabla K\|_{L^{\infty}} |{\mathbb T}|^{1/2} n \sup_{0<t\leq T}\left(\|v_{1i}-v_{2i}\|_{L^2}  \sum_{j=1}^{N} \|v_{1j} \|_{L^2} \right. \\
		& \hspace*{1cm} \left. + \|v_{2i}\|_{L^2}  
		\sum_{j=1}^{N} \| v_{1j}- v_{2j}\|_{L^2} \right),
		\end{align*} 
		where $ \|H\|_{\infty}=\max_{i,j} |h_{ij}| $. In the second inequality we used H\"older's inequality, and in the last one we used \cref{l2:h1b}. 
		From the previous estimate, we obtain
		\begin{align*}	
		\|Qv_{1}-Qv_{2}\|_{L^2} =& \sum_{i=1}^{N} 	\|Qv_{1i}-Qv_{2i}\|_{L^2} \\
		&\leq 2 C\sqrt{T}\|H\|_\infty \|\nabla K\|_{L^{\infty}} |{\mathbb T}|^{1/2} n\sup_{0<t\leq T}\left(\sum_{i=1}^{N} \|v_{1i}-v_{2i}\|_{L^2}  \sum_{j=1}^{N} \|v_{1j} \|_{L^2} \right. \\
		& \hspace*{1cm} \left. + \sum_{i=1}^{N}  \|v_{2i}\|_{L^2}  \sum_{j=1}^{N} \| v_{1j}- v_{2j}\|_{L^2} \right) \\
		&\leq 2 C\sqrt{T}\|H\|_\infty \|\nabla K\|_{L^{\infty}} |{\mathbb T}|^{1/2} n\sup_{0<t\leq T}\left( \|v_{1}-v_{2}\|_{L^2} ( \|v_{1} \|_{L^2} + \|v_{1} \|_{L^2})\right) \\
		& \leq 4 M C\sqrt{T}\|H\|_\infty \|\nabla K\|_{L^{\infty}} |{\mathbb T}|^{1/2} n\sup_{0<t\leq T} \|v_{1}-v_{2}\|_{L^2}.
		\end{align*} 
		The last inequality is obtained from $v_1,v_2 \in L^\infty((0,T_{min}), B_M(0))^N$, so $\|v_1\|_{L^2},\|v_2\|_{L^2}\leq M$.
		For 
		\[ T< T_2 :=\frac{1}{|\mathbb T|(4MCn\| H \|_{\infty}\|\nabla K\|_{L^\infty})^2} \] 
		we have 
		\[\sup_{0<t\leq T} \|Qv_1-Qv_2\|_{L^2} < \sup_{0<t\leq T}\| v_1 - v_2\|_{L^2},\]
		which means $ Q v_1 - Q v_2 \in L^\infty((0,T_2), B_M(0))^N $.  Thus $Q$ is a strict contraction in $ L^\infty((0,T_{min}), B_M(0))^N $, where we can finally define $T_{min}$ as
		\[ T_{min} := \min\left\{ T_1, T_2\right\}.\]
		\underline{{\bf Step 3:}} The previous argument also shows that $Q$ is Lipschitz continuous, hence, by the Banach fixed point theorem, $Q$ has a unique fixed point for $T<T_{min}$. This fixed point is a mild solution of (\ref{eq:model2}) and it satisfies 
		\[ u\in L^\infty((0,T), L^2(\mathbb{T} ))^N\]
		for $T<T_{min}$. 
		The mild solution automatically satisfies the initial condition:
		\[ \lim_{t\to 0} u(x,t) = u_0(x).\]
	\end{proof}
	
	%%%%%%%%%%%%%%%%%%%%
	\subsection{Global existence in time}
	Let $ u $ be a mild solution of \Cref{eq:model3}. Our strategy moving forward will be to show that, for the period of time that $\|u\|_{L^1}$ remains bounded, solutions exist and grow at most exponentially in $L^2$.  We will then show that the statement `$\|u\|_{L^1}$ is unbounded' leads to a contradiction.
	
	With this in mind, we define a time $T_*$ as follows: if $\|u\|_{L^1}$ is bounded for all time, then let $T_*=\infty$.  Otherwise, $\|u\|_{L^1}\rightarrow \infty$ as $t \rightarrow T_{max}$ for some $T_{max} \in (0,\infty]$, so let $T_*$ be the earliest time at which $\|u\|_{L^1}=2\|u_{0}\|_{L^1}$.  Our objective will be to show that the case where $\|u\|_{L^1}\rightarrow \infty$ as $t \rightarrow T_{max}$  leads to a contradiction when $n=1$ (one spatial dimension), so that $\|u\|_{L^1}$ is bounded for all time.  This will enable us to prove that the solution from \cref{thm:local} is global in time when $n=1$.
	\begin{lemma}	
		\label{l:gradl1}
		Let $ u =(u_1,\dots,u_N)$ be a mild solution and ${ K}:\mathbb{T}\rightarrow {\mathbb R}$ be differentiable with $\nabla {K} \in L^\infty(\mathbb{T})$. Then there exists a constant $ \nu_i $ such that $\| \nabla (\mathcal K\ast u_i) \|_{L^{\infty}} \leq \nu_i$ for all $t<T_*$, $i\in \{1,\dots,N\}$.  If $\nu=\nu_1+\dots+\nu_N$ then $\| \nabla (\mathcal K\ast u) \|_{L^{\infty}} \leq \nu$.
	\end{lemma}
	\begin{proof}
		Applying Young's convolution inequality, we have $ \| \nabla (\mathcal K\ast u_i) \|_{L^{\infty}} \leq \|\nabla \mathcal K \|_{L^{\infty}} \|u_i\|_{L^1} $. By the definition of $T_*$, $\|u_i\|_{L^1}(t)$ is bounded for $ t<T_* $.  Thus there exists a constant $ \nu_i $ such that  $ \|\nabla \mathcal K \|_{L^{\infty}} \|u_i\|_{L^1} \leq \nu_i$.  The result $\| \nabla (\mathcal K\ast u) \|_{L^{\infty}} \leq \nu$ follows from the definitions of $\nu$ and the norm on $(L_1)^N$. 
	\end{proof}
	
	%		{\color{red}[JRP: To fix the issue with Equation (6), I think we can do the following.  First define a time $T_*$ as follows.  If $\|u_i\|_{L^1}$ is bounded for all time and each $i$, then let $T_*=\infty$.  Otherwise, let $T_*$ be the earliest time at which $\|u_i\|_{L^1}=2\|u_{i,0}\|_{L^1}$ for some $i$ (which exists due to the time-continuity of $\|u_i\|_{L^1}$ and the fact that $\|u_i\|_{L^1}\rightarrow \infty$ at some point, since it is not bounded).  Then the Lemmas 5-6 are true only for $t<T_*$.  However, this does mean that the solution is positive up to time $T_*$, by Lemma 7.  Now, if $T_*$ is finite then this means that $\|u_i\|_{L^1}$ will be strictly greater than $\|u_{i,0}\|_{L^1}$ for some $t_*<T_*$.  But, since $\int_{\mathbb T} u_i dx=\|u_{i,0}\|_{L^1}$ for all time, this implies that there must be some $x$ such that $u_i(x,t_*)<0$, contradicting positivity.  Thus we must have $T_*=\infty$ and solutions are global in time.]}
	\begin{lemma} 
		\label{lem:local}
		Assume $u_0\in H^2(\mathbb{T} )^N$. Then the mild solution from \cref{thm:local} satisfies
		\[ u\in C^1((0,T_*), L^2(\mathbb{T} ))^N \cap C^0([0,T_*), H^2(\mathbb{T} ))^N\]
		In one spatial dimension this implies 
		\[ u\in C^1((0,T_*), L^2(\mathbb{T} ))^N \cap C^0([0,T_*), C^2(\mathbb{T} ))^N,\]
		and mild solutions are classical up to time $T_*$. 
	\end{lemma}
	\begin{proof} As we are dealing with a system of equations $u=(u_1, \dots, u_N)$, we consider each component separately. 
		For each of the components $u_i$ for $i=1,\dots,N$ we multiply the $i$-th row of \Cref{eq:model3} by $u_i$ and integrate:
		\begin{align*}
		\frac{1}{2}\frac{d}{dt} \|u_i\|_{L^2}^2 &= \int_\mathbb{T}  u_i u_{it} dx\\
		& = \int_\mathbb{T}  D_i u_i \Delta u_i dx - \int_\mathbb{T}  u_i\nabla\cdot  (u_i \nabla((H\bar u)_i))  dx \\
		%		& = \int_\mathbb{T} D_i \nabla \cdot (u_i \nabla u_i) dx -\int_\mathbb{T}  D_i|\nabla u_i|^2 dx - \int_\mathbb{T}  \nabla \cdot (u_i^2 \nabla ((H\bar u)_i)) dx + \int_\mathbb{T} u_i  \nabla u_i \cdot \nabla ((H\bar u)_i) dx \\
		& =  -\int_\mathbb{T}  D_i|\nabla u_i|^2 dx + \int_\mathbb{T} u_i  \nabla u_i \cdot \nabla ((H\bar u)_i) dx \\
		& =  -\int_\mathbb{T}  D_i\sum_{h=1}^{n} (\partial_{x_h} u_i)^2 dx + \int_\mathbb{T} u_i \sum_{h=1}^{n} (\partial_{x_h} u_i) \partial_{x_h} ((H\bar u)_i) dx \\
		%		& =  -\int_\mathbb{T}  D_i\sum_{h=1}^{n} (\partial_{x_h} u_i)^2 dx\\
		%		& \hspace*{1cm} + \int_\mathbb{T} u_i \partial_{x_1}\left( u_i \partial_{x_1} ((H\bar u)_i)\right)dx + \dots + \int_\mathbb{T} u_i \partial_{x_n}\left( u_i \partial_{x_n} ((H\bar u)_i)\right) dx \\	
		%		& \leq  -\int_\mathbb{T}  D_i\sum_{h=1}^{n} (\partial_{x_h} u_i)^2 dx \\
		%		& \hspace*{1cm}+ \| \partial_{x_1} ((H\bar u)_i)\|_{L^{\infty}} \int_\mathbb{T} |u_i \partial_{x_1} u_i| dx + \dots +  \| \partial_{x_n} ((H\bar u)_i)\|_{L^{\infty}} \int_\mathbb{T} |u_i \partial_{x_n} u_i| dx  \\	
		& \leq \sum_{h=1}^{n} \left( -\int_\mathbb{T}  D_i(\partial_{x_h} u_i)^2 dx  + \| \partial_{x_h} ((H\bar u)_i)\|_{L^{\infty}} \int_\mathbb{T} |u_i \partial_{x_h} u_i| dx	\right)\\
		& = -\int_\mathbb{T}  D_i\left|\nabla u_i\right|^2 dx  + \| \nabla ((H\bar u)_i)\|_{L^{\infty}} \int_\mathbb{T} |u_i \nabla u_i| dx	\\
		& = -\int_\mathbb{T}  D_i\left|\nabla u_i\right|^2 dx  + \|\sum_{j=1}^{N} h_{ij} \nabla(K \ast u_j)\|_{L^{\infty}} \int_\mathbb{T} |u_i \nabla u_i| dx	\\
		& \leq -\int_\mathbb{T}  D_i\left|\nabla u_i\right|^2 dx  +  \| H\|_{\infty} \sum_{j=1}^{N} \| \nabla(K \ast u_j)\|_{L^{\infty}} \int_\mathbb{T} |u_i \nabla u_i| dx	\\
		& \leq  -\int_\mathbb{T}  D_i\left|\nabla u_i\right|^2 dx  + \| H\|_{\infty} \nu \int_\mathbb{T} |u_i \nabla u_i| dx	\\
		& \leq  \left(- D_i+\frac{\varepsilon}{2} \left(\|H\|_{\infty}   \nu\right)^2\right)\int_\mathbb{T} \left|\nabla u_i\right|^2 dx  + \frac{n}{2\varepsilon} \int_\mathbb{T} |u_i|^2 dx\\
		\end{align*}
		where $ \|H\|_{\infty}=\max_{i,j} |h_{i,j}| $. In the third equality we used integration by parts and the periodic boundary conditions, the first inequality uses H\"older's inequality, the third inequality uses \Cref{l:gradl1}, which is valid for $t<T_*$, and the fourth inequality uses Young's inequality.
		
		Now we choose $\ep$ such that $- D_i+\frac{\varepsilon}{2} \left(\|H\|_{\infty}  \nu\right)^2<0$ for all $i,j=1,\dots,N$ so that 
		\[  \frac{1}{2}\frac{d}{dt} \|u_i\|_{L^2}^2 \leq \frac{n}{2\ep} \|u_i\|_{L^2}^2. \]
		Applying Gr\"onwall's Lemma, we find 
		\[ \|u_i\|_{L^2} \leq \|u_{i0}\|_{L^2} e^{\frac{{n}t}{2\ep}}.\]
		Finally, we observe that
		\[\sum_{i=1}^{N} \|u_i\|_{L^2} \leq\sum_{i=1}^{N} \|u_{i0}\|_{L^2} e^{\frac{{n}t}{2\ep}},\]
		from which we obtain
		\begin{align}
		\label{eq:L2est}
		\|u\|_{L^2} \leq \|u_{0}\|_{L^2} e^{\frac{{n}t}{2\ep}}.
		\end{align}
		Hence solutions exist and grow at most exponentially in $L^2$ up to time $T_*$. 
		\\
		\\
		Now we find an estimate in $H^1$ for each component $u_i$, $i=1,\dots, N$:
		%	\begin{align*}
		%		\frac{1}{2}\frac{d}{dt}  \|\nabla u_i\|_{L^2}^2 =& - \int_\mathbb{T}  (\nabla u_{it})\cdot(\nabla u_i) dx \\
		%		&= - \int_\mathbb{T}  u_{it} \Delta u_i dx \\
		%		&= - \int_\mathbb{T}  D_i (\Delta u_i)^2 dx +  \int_\mathbb{T} \Delta u_i \nabla\cdot (u_i \nabla ((H \bar u)_i)) dx\\
		%		& = -\sum_{s=1}^{n}  \int_\mathbb{T}  D_i (\partial_{x_s}^2 u_i)^2 dx + \sum_{s=1}^{n}  \int_\mathbb{T} \partial_{x_s}^2 u_i \left( 
		%		\sum_{h=1}^{n} \partial_{x_h} \left( u_i \partial_{x_h} \sum_{j=1}^{N} h_{ij} K\ast u_j\right) \right) dx\\
		%		& \leq \left( -D_i + \frac{\varepsilon}{2} \right) \sum_{s=1}^{n}  \int_\mathbb{T} (\partial_{x_s}^2 u_i)^2 dx + \frac{1}{2 \varepsilon} 	\sum_{h=1}^{n} \int_\mathbb{T}\left(  \partial_{x_h} \left( u_i \partial_{x_h} \sum_{j=1}^{N} h_{ij} K\ast u_j\right) \right)^2 dx, 
		%	\end{align*}
		
		%In the above, you have commuted the sums with squares in several places, so that $\sum_{s=1}^{n}  \int_\mathbb{T} (\partial_{x_s}^2 u_i)^2 dx$ should be $  \int_\mathbb{T} \left(\sum_{s=1}^{n}\partial_{x_s}^2 u_i\right)^2 dx$.  But this can be fixed simply, as you can just write it all out without using summand notation:
		\begin{align*}
		\frac{1}{2}\frac{d}{dt}  \|\nabla u_i\|_{L^2}^2 =& - \int_\mathbb{T}  (\nabla u_{it})\cdot(\nabla u_i) dx \\
		&= - \int_\mathbb{T}  u_{it} \Delta u_i dx \\
		&= - \int_\mathbb{T}  D_i (\Delta u_i)^2 dx +  \int_\mathbb{T} \Delta u_i \nabla\cdot (u_i \nabla ((H \bar u)_i)) dx\\
		&= \left(-D_i+\frac{\ep_2}{2}\right) \int_\mathbb{T} (\Delta u_i)^2 dx +  \frac{1}{2\ep_2}\int_\mathbb{T} ( \nabla\cdot (u_i \nabla ((H \bar u)_i)))^2 dx,
		\end{align*}
		where we used Young's inequality to obtain the last estimate. We now chose $\ep_2>0$ small enough such that  $-D_i + \frac{\ep_2}{2}<0 $ for every $i=1,\dots,N$. We then continue the previous estimate as 
		\begin{align*}
		\frac{1}{2}\frac{d}{dt} \|\nabla u_i\|_{L^2}^2 &\leq  \frac{1}{2\ep_2}\|\nabla \cdot(u_i\nabla((H\bar{u})_i))\|_{L^2}^2 \\
		& = \frac{1}{2\ep_2}\left\|\sum_{h=1}^n \partial_{x_h} \left(u_i \partial_{x_h} \sum_{j=1}^N h_{ij} K\ast u_j\right)\right\|_{L^2}^2 \\
		&\leq \frac{1}{2\ep_2}\left\|\sum_{h=1}^n (\partial_{x_h} u_i) \partial_{x_h} \sum_{j=1}^N h_{ij} K\ast u_j+\sum_{h=1}^n u_i \partial_{x_h}^2 \sum_{j=1}^N h_{ij} K\ast u_j\right\|_{L^2}^2 \\
		&\leq \frac{1}{2\ep_2}\left(\left\|\sum_{h=1}^n (\partial_{x_h} u_i) \partial_{x_h} \sum_{j=1}^N h_{ij} K\ast u_j\right\|_{L^2}+\left\|\sum_{h=1}^n u_i \partial_{x_h}^2 \sum_{j=1}^N h_{ij} K\ast u_j\right\|_{L^2}\right)^2 \\
		&\leq \frac{1}{\ep_2}\left(\left\|\sum_{h=1}^n (\partial_{x_h} u_i) \partial_{x_h} \sum_{j=1}^N h_{ij} K\ast u_j\right\|_{L^2}^2+\left\|\sum_{h=1}^n u_i \partial_{x_h}^2 \sum_{j=1}^N h_{ij} K\ast u_j\right\|_{L^2}^2\right) \\
		&\leq \frac{1}{\ep_2}\left(\sum_{h=1}^n \|\partial_{x_h} u_i\|_{L^2}  \sum_{j=1}^N |h_{ij}| \|\partial_{x_h}(K\ast u_j)\|_{L^\infty}\right)^2\\
		&\quad+\frac{1}{\ep_2}\left(\|u_i\|_{L^2}  \sum_{j=1}^N \sum_{h=1}^n |h_{ij}| \|\partial_{x_h}^2(K\ast u_j)\|_{L^\infty}\right)^2 \\
		\end{align*}
		\begin{align*}
		&\leq \frac{1}{\ep_2}\left(\|\nabla u_i\|_{L^2} \|H\|_\infty \sum_{j=1}^N  \|\nabla(K\ast u_j)\|_{L^\infty}\right)^2\\
		&\quad+\frac{1}{\ep_2}\left(\|u_i\|_{L^2} \|H\|_\infty \sum_{h=1}^n \sum_{j=1}^N  \|(\partial_{x_h}K)\ast (\partial_{x_h}u_j)\|_{L^\infty}\right)^2 \\
		&\leq \frac{1}{\ep_2}\left(\|\nabla u_i\|_{L^2} \|H\|_\infty \sum_{j=1}^N  \|\nabla K\|_{L^\infty} \|u_j\|_{L^2}|{\mathbb T}|^{1/2}\right)^2\\
		&\quad+\frac{1}{\ep_2}\left(\|u_i\|_{L^2} \|H\|_\infty \sum_{h=1}^n \sum_{j=1}^N  \|\partial_{x_h}K\|_{L^\infty} \|\partial_{x_h}u_j\|_{L^1}\right)^2 \\
		&\leq \frac{1}{\ep_2}\left(\|\nabla u_i\|_{L^2} \|H\|_\infty |{\mathbb T}|^{1/2}\sum_{j=1}^N  \|\nabla K\|_{L^\infty} \|u_j\|_{L^2}\right)^2\\
		&\quad+\frac{1}{\ep_2}\left(\|u_i\|_{L^2} \|H\|_\infty |{\mathbb T}|^{1/2} \sum_{j=1}^N  \|\nabla K\|_{L^\infty} \|\nabla u_j\|_{L^2}\right)^2 \\
		&\leq \frac{1}{\ep_2} \|H\|_\infty^2 |{\mathbb T}|\|\nabla K\|_{L^\infty}^2\left(\|\nabla u_i\|_{L^2}^2\left(\sum_{j=1}^N   \|u_j\|_{L^2}\right)^2+\|u_i\|_{L^2}^2\left(\sum_{j=1}^N   \|\nabla u_j\|_{L^2}\right)^2\right)\\
		&\leq \frac{N}{\ep_2} \|H\|_\infty^2 |{\mathbb T}|\|\nabla K\|_{L^\infty}^2\left(\|\nabla u_i\|^2_{L^2}\sum_{j=1}^N   \|u_j\|_{L^2}^2+\|u_i\|^2_{L^2}\sum_{j=1}^N   \|\nabla u_j\|_{L^2}^2\right),
		\end{align*}
		in which we have used Young's inequality in the fourth inequality, \cref{l2:h1b} in the seventh, \cref{l:deltainfty} in the eighth, and Young's inequality in the ninth.
		\\
		\\
		Taking the sum over all the components $i \in \{1,\dots,N\}$, we have
		\begin{align*}
		\frac{1}{2}\frac{d}{dt} \sum_{i=1}^{N} \|\nabla u_i\|_{L^2}^2 
		& \leq \frac{2N}{\ep_2} \|H\|_\infty^2 |{\mathbb T}|\|\nabla K\|_{L^\infty}^2\sum_{i=1}^N\|\nabla u_i\|_{L^2}^2\sum_{j=1}^N \|u_j\|_{L^2}^2.
		\end{align*}
		By defining
		\begin{align*}
		A=\frac{4N}{\ep_2} \|H\|_\infty^2 |{\mathbb T}|\|\nabla K\|_{L^\infty}^2\sum_{j=1}^N \|u_{0j}\|_{L^2}^2,
		\end{align*}
		and using \eqref{eq:L2est}, we arrive at
		\begin{align*}
		\frac{1}{2}\frac{d}{dt} \sum_{i=1}^{N} \|\nabla u_i\|_{L^2}^2 	& \leq \frac{A}{2} e^{\frac{nt}{2\ep}}\sum_{i=1}^N \|\nabla u_i\|_{L^2}^2.
		\end{align*}
		Applying Gr\"onwall's Lemma, we have
		\begin{equation*}
		\sum_{i=1}^{N} \|\nabla u_i\|_{L^2}^2(t) \leq  \sum_{i=1}^{N} \|\nabla u_{i0}\|_{L^2}^2 \exp\left(A\int_0^t \exp\left({\frac{ns}{2\ep}}\right)ds\right),
		\end{equation*}
		for each time $t<T_*$.  Thus solutions remain bounded in $H^1({\mathbb T})$ until time $T_*$.
		
		%Finally we obtain
		%\begin{align*}
		%	\frac{d}{dt} \frac{1}{2} \sum_{i=1}^{N} \|\nabla u_i\|_{L^2}^2 
		%	& \leq \frac{1}{\ep}|\mathbb{T}| \|H\|_{\infty}^2 \| \nabla K\|_{L^{\infty}}^2 \sum_{i=1}^{N} \| u_i\|_{L^2}^2 \| \nabla u_i\|_{L^2}^2.
		%\end{align*}
		%We saw already in Lemma \ref{l:global} that $\|u_i\|_{L^2}^2$ is exponentially bounded. Then, by Gronwall's Lemma
		%\begin{equation*}
		%\sum_{i=1}^{N} \|\nabla u_i\|_{L^2}^2 \leq  \sum_{i=1}^{N} \|\nabla u_{i0}\|_{L^2}^2 e^{c t} ,
		%\end{equation*}
		%where c is a constant. Finally, observing that 
		%\begin{equation*}
		%\|\nabla u_i\|_{L^2}^2 \leq  \sum_{i=1}^{N} \|\nabla u_{i0}\|_{L^2}^2 e^{c t} \leq % \|\nabla u_{0}\|_{L^2}^2  e^{c t},
		%\end{equation*}
		% we obtain 
		%\begin{equation*}
		%	\|\nabla u\|_{L^2}= \sum_{i=1}^{N} \|\nabla u_i\|_{L^2} \leq N \|\nabla u_{0}\|_{L^2} e^{c t/2},
		%\end{equation*}
		%for each time $t<T_*$ 
		Now let us consider the claim:
		\[ u\in \underbrace{C^1((0,T_*), L^2(\mathbb{T} ))^N}_{(I)} \cap \underbrace{C^0([0,T_*), H^2(\mathbb{T} ))^N}_{(II)}.\]
		Looking again at the mild formulation in \Cref{eq:mild}, we have that $u\in H^1$, $\nabla(H\bar u)\in H^1$ and the integral term is in $H^1$. The first term involves the heat equation semigroup and the initial condition, and by the classical theory of the linear heat equation, the term $e^{D\Delta t} u_0$ is in $H^1$ and differentiable in time. Hence also $u_t$ exists and is in $L^2$. This explains (I). 
		Finally, writing down the equation once more:
		\[ u_t = D\Delta u - \nabla\cdot (u \nabla\cdot(H\bar u))\]
		we now know that $u_t$ is in $L^2$ and the non-local term as well. Hence $\Delta u\in L^2$, which implies (II). 
		%	Finally, in one space dimension, we have the Sobolev embedding from $H^2$ to $C^1$, and the spatial derivatives are classical. 
		
		In one spatial dimension, we also have the Sobolev embedding from $H^2$ to $C^1$. Indeed, we can use this to show that solutions are in $C^2$ for $n=1$.  First note that \[((H\bar{u})_i)_x=\sum_{j=1}^N h_{ij}\frac{\partial K}{\partial x}\ast u_i,\] and \[((H\bar{u})_i)_{xx}=\sum_{j=1}^N h_{ij}\frac{\partial K}{\partial x}\ast \frac{\partial u_i}{\partial x},\]
		which are both continuous.  Therefore $[u_i((H\bar{u})_i)_x]_x=u_{ix}((H\bar{u})_i)_x+u_{i}((H\bar{u})_i)_{xx}$ is continuous.  It follows from the mild formulation in \Cref{eq:mild} that $u_{it}$ is continuous.  Consequently, $D_iu_{ixx}=u_{it}+[u_i((H\bar{u})_i)_x]_x$ is continuous, so $u_i$ is in $C^2({\mathbb T})$ (where ${\mathbb T}=[0,L]$ here, since $n=1$).
	\end{proof}

	\begin{lemma}
		\label{lem:pos2}Consider the solution from \cref{lem:local} in one spatial dimension, so that $n=1$, $\mathbb{T}=[0,L]$, and $ u\in C^1((0,T_*), L^2(\mathbb{T} ))^N \cap C^0((0,T_*), C^2(\mathbb{T} ))^N $. Let $u_0\in C^2(\mathbb{T} )^N$ such that $u_0(x)>0$ for $x \in \mathbb{T} $.  Then $u(x,t)>0$ for $x \in \mathbb{T} $ and $ t<T_* $. 
	\end{lemma}
	\begin{proof}
		We let $u=(u_1,\dots,u_N)$ and work with each component separately. Assume that there is a first time $t_0>0$ such that the solution for $u_i$ becomes zero at a point $x_0$. We can rule out the case that $u_i(t_0, x) \equiv 0$, since the system (\ref{eq:model3}) conserves total mass. Then we have
		\[ u(t_0, x_0) =0, \quad u_{ix}(t_0, x_0) =0,\quad u_{ixx}(t_0, x_0) >0, \quad u_{it}(t_0, x_0)<0.\]
		System (\ref{eq:model1}) evaluated at $(t_0, x_0)$ becomes 
		\begin{eqnarray*}
			\underbrace{u_{it}(t_0, x_0)}_{<0} &=& D_i u_{ixx}(t_0, x_0) - [u_i(t_0,x_0) ((H\bar{u})_i(t_0,x_0))_x]_x\\
			&=& \underbrace{D_i u_{ixx}(t_0, x_0)}_{>0} - \Bigl[\underbrace{u_{ix}(t_0,x_0)}_{=0} ((H\bar{u})_i(t_0,x_0))_x + \underbrace{u(t_0, x_0)}_{=0} (H\bar{u}(t_0,x_0))_{ixx}\Bigr],\end{eqnarray*}
		leading to a contradiction. Hence $u_i(x,t)>0$.
	\end{proof}
	
	\begin{theorem}\label{thm:global}
		Let $u_0\in C^2(\mathbb{T} )^N$ such that $u_0(x)>0$ for $x \in \mathbb{T}$.  Then the solution from \Cref{lem:local} is global in time (i.e. $T_*=\infty$) when working in one spatial dimension ($n=1$).
	\end{theorem}
	\begin{proof}
		Recall that if $T_*<\infty$ then $\|u\|_{L^1}\rightarrow \infty$ at some point in time and $T_*$ defined as the earliest time at which $\|u\|_{L^1}=2\|u_{0}\|_{L^1}$.  Therefore  $\|u\|_{L^1}$ will be strictly greater than $\|u_{0}\|_{L^1}$ for some $t_* \in (0,T_*)$.  But, since $\int_{\mathbb T} u dx=\|u_{0}\|_{L^1}$ for all time, we have $\int_{\mathbb T} u(x,t_*) dx < \int_{\mathbb T} |u(x,t_*)|dx$, which implies that there must be some $x$ such that $u(x,t_*)<0$, contradicting positivity (\cref{lem:pos2}).  Thus we must have $T_*=\infty$ and solutions are global in time.
	\end{proof}
	
	\section{Numerics}
	\label{sec:numeric}
	In this section we describe a method to solve  System \eqref{eq:model1} numerically, based on the general class of spectral methods \cite{canuto2007spectral}.  For simplicity, we focus on simulations within 1D domains.  However, this procedure may be also extended to any spatial dimension.  Although our analytic results rely on the averaging kernel, $K$, being twice differentiable, our numerical method does not rely on this constraint.  Since the study of \cite{pottslewis2019} used a top-hat kernel (which is not differentiable), we demonstrate our method using this kernel as well as an example twice-differentiable kernel.
	
	The leading idea behind a spectral method is to write the solution of a PDE as a sum of smooth basis  functions with time dependent coefficients. By substituting this expansion in the PDE, we obtain a system of ordinary differential equations (ODEs), which can be solved using any numerical method for ODEs \cite{butcher2008numerical}.
	
	In the previous section we showed that, under the hypothesis of \cref{lem:local}, any solution $ u(x,t) $ to System \eqref{eq:model1} is $C^2$-smooth, so it is possible to expand it as
	$$ u (x,t)= \sum_{h=-\infty}^{\infty} \hat{u}_{h} (t) \phi_h (x),$$
	where the coefficients $ \hat{u}_{h} $ are computed by using the global behaviour of the function $ u $ and $ \{\phi_h\}_h $ is a complete set of orthogonal smooth functions.
	
	Since System \eqref{eq:model1} is periodic in space with period $ L $, we adopt the Fourier basis as complete set of orthogonal functions and expand each component of the solution $ u=(u_1, \dots, u_N) $ as
	\begin{equation}\label{eq:FourierExp}
	u_j (x,t)= \sum_{h=-\infty}^{\infty} \hat{u}_{jh} (t)  e^{\frac{2 \pi i}{L}h x},  \text{ for }  j=1, \dots, N,
	\end{equation}
	where $ \hat{u}_{jh}(t) = \frac{1}{L}\int_{0}^{L} u_j (x,t) e^{-\frac{2 \pi i}{L}h x} dx$ are the Fourier coefficients, which represent the solution in the frequency space.
	
	One of the advantages of working with the Fourier expansion is that the operation of derivation becomes particularly simple if performed in the frequency space. Indeed, differentiating \Cref{eq:FourierExp}, we find
	\begin{equation}
	\partial_x	u_j (x,t)= \sum_{h=-\infty}^{\infty} \frac{2 \pi i}{L}h \hat{u}_{jh} (t)  e^{\frac{2 \pi i}{L}h x},  \text{ for }  j=1, \dots, N,
	\end{equation}
	we see that the Fourier coefficients of the derivative are obtained by multiplying each $ \hat{u}_{jh} $ by the term $ \frac{2 \pi i}{L}h $.
	
	Another important property of the Fourier transform, particularly useful in our case, is that the convolution in the physical space is equivalent to a multiplication in the frequency space. Indeed, the Convolution Theorem states that the convolution between two functions  $ f (x)= \sum_{h=-\infty}^{\infty} \hat{f}_{h}  e^{\frac{2 \pi i}{L}h x} $ and $ g (x)= \sum_{h=-\infty}^{\infty} \hat{g}_{h}  e^{\frac{2 \pi i}{L}h x} $ has the following Fourier expansion
	\begin{equation}
	\label{eq:conv}
	f \ast g (x) = \sum_{h=-\infty}^{\infty} \hat{f}_{h}  \hat{g}_{h} e^{\frac{2 \pi i}{L}h x}.
	\end{equation}
	Therefore, to solve numerically System \eqref{eq:model1} the operations of differentiations and convolution will be performed in the frequency space, while multiplications will be done in the physical space.
	
	To implement our numerical method, we discretize both spatial and temporal domain, and consider the approximation of the solution $ u(x,t) $ on the grid points $ x_m = m \Delta x $ and $ t^n= n \Delta t $, with $ m \in \{0,1,\dots,M-1\} $ and $ n \in \mathbb{N} $. We define $ U_{jm}^{n} = u_j (x_m, t^n) $.  Then, in discrete space, the coefficients $  \hat{u}_{jh}(t) $ of \Cref{eq:FourierExp} are replaced by
	\begin{equation}\label{eq:dft}
	\hat{U}_{jh}^n =\frac{1}{M}\sum_{m=0}^{M-1}  U_{jm}^{n}  e^{-\frac{2 \pi i}{M}h m}, 		
	\end{equation}
	which represent the discrete Fourier transform (DFT) of $ u_j(x,t) $.
	
	The inverse discrete Fourier transform (IDFT), used to compute  $  U_{jm}^{n} $ from $ \hat{U}_{jh}^n $, is given by the formula
	\begin{eqnarray}\label{eq:idft}
	U_{jm}^{n} = \sum_{h=0}^{M-1} \hat{U}_{jh}^{n} e^{\frac{2 \pi i}{M}h m}.
	\end{eqnarray}
	We can convert the solution from physical to frequency space, and vice versa, using the relations \eqref{eq:dft} and \eqref{eq:idft}. However, we can speed the procedure up considerably by using a Fast Fourier Transform (FFT) algorithm, which reduces the number of computations from $ M^2 $ to $ M \log M $ \cite{press2007numerical}. Analogously, an Inverse Fast Fourier Transform (IFFT) algorithm can be used to perform a fast backward Fourier transform from the frequency domain to the physical domain.
	
	Let $ \mathbf{U}_j^n=[U_{j 1}^n, \dots, U_{j (M-1)}^n] $, for $ j=1, \dots,N $ and $ \hat{\mathbf{U}}_j^n=[\hat{U}_{j1}^n, \dots, \hat{U}_{j(M-1)}^n] $, for $ j=1, \dots,N $, which represent the solution in the frequency domain at time $ t=n \Delta t $.  Then the algorithm for calculating the solution is as follows.
	
	First, we calculate the  non-local terms $ \bar{\mathbf{U}}_j^n = K \ast \mathbf{U}_j^n $ by passing to the frequency domain and applying the Convolution Theorem (\Cref{eq:conv}).  We then stay in the frequency domain to calculate the derivative $ \partial_x \bar{\mathbf{U}}_j^n $.  Passing back to physical space, we calculate the product  $ \mathbf{U}_i^n \cdot \partial_x  \bar{\mathbf{U}}_j^n  $.  Then the derivative of this product, $ \partial_x(\mathbf{U}_i^n \cdot \partial  \bar{\mathbf{U}}_j^n)$, is calculated in the frequency domain.  This deals with the second term in our PDE (System \ref{eq:model1}).	Finally, we calculate the diffusion term from System (\ref{eq:model1}) by passing to frequency space.  
	
	This whole procedure results in defining a function, $f(\mathbf{U}_j^n)$, which is a discrete representation of the right-hand side of the PDE in  System (\ref{eq:model1}).  Thus we have the following system of ODEs
	\begin{eqnarray}
	\frac{d \mathbf{U}_j^n}{dt} = f(\mathbf{U}_j^n), \quad j=1, \dots, N,
	\end{eqnarray}
	which can be solved using any ODE solver.  In particular, we used a Runge-Kutta scheme with $ \Delta t = 10^{-4} $ \cite{butcher2008numerical}.  To calculate the coefficients of Fourier transform and inverse Fourier transform, we used the \texttt{drealft} fast Fourier transform subroutine from \cite{press2007numerical}. This routine requires that the number of grid points must be a power of 2.  We used the spatial domain $ [0,1] $ with $ 128 $ spatial grid points (so $\Delta x=1/128$) and periodic boundary conditions.
	
	For the spatial averaging kernel $ K $, we used two different functions.  The first is the von Mises distribution
	\begin{equation}\label{eq:smoothk}
	K_a(x)= \frac{e^{a \cos(2 \pi x)}}{I_0(a)},
	\end{equation}
	defined on $[-1/2,1/2]$ (which is equivalent to $[0,1]$ due to the periodic boundary conditions), where $ I_0(a) $ is the modified Bessel function of order $ 0 $. This distribution both satisfies the periodic boundary conditions and is twice differentiable, as required by \cref{l2:h1}, \cref{l2:h1b}, \cref{l:deltainfty} and \cref{l:gradl1}.  We compare this with the following top-hat function on $[-1/2,1/2]$, used by \cite{pottslewis2019}
	\begin{equation}\label{eq:nonsmoothk}
	K_\gamma (x)= \begin{cases}
	\frac{1}{2 \gamma}, & -\gamma\leq x \leq \gamma, \\
	& \\
	0, & \text{otherwise}.
	\end{cases}
	\end{equation}
	To compare numerical solutions with the two averaging kernels, $K_a$ and $K_\gamma$, we use a common standard deviation
	\begin{equation}\label{eq:sd}
	\sigma=\sqrt{\int_{-1/2}^{1/2} x^2K(x){\rm d}x-\left(\int_{-1/2}^{1/2} xK(x){\rm d}x\right)^2}.
	\end{equation}
	We implemented our algorithm in the C programming language and demonstrate it using the simple case of two interacting populations, $ u_1 $ and $ u_2 $.  
	
	In \cref{fig:StationarySolutions} we show the spatiotemporal evolution of the numerical solution, with $K=K_a$, for different values of the standard deviation $ \sigma $. For $ \sigma=0.1 $, we used a smooth random perturbation of the homogeneous steady state as initial condition. In this case, the solution appears to evolve towards a stationary state, and we stopped the numerics when the difference between two time-consecutive solutions went below $10^{-6}$.  This took about 5 seconds of computational time to reach. %We have verified that the solution has reached its steady configuration by ascertaining that the difference between two consecutive in time configurations is zero.
	We then used this stationary state as initial condition for a simulation with $ \sigma=0.05 $, whose spatiotemporal evolution is shown in the second line of \cref{fig:StationarySolutions}. As in the previous case, the solution appears to settle into a stationary state, which was used as initial condition to perform a simulation with $ \sigma=0.025 $.
	We see that, as $\sigma$ is decreased, the steady state solutions become increasingly flat-topped.  
	
	In each of these examples, $h_{ii}=0$ for $i=1,2$.  In this case, \cite{pottslewis2019} showed that the system admits an energy functional which decreases over time, a feature that often accompanies systems that reach a stable steady state, and indeed this is what we observe in our numerics.  However, if we drop the $h_{ii}=0$ assumption, it is possible to observe patterns that exhibit oscillatory behaviour that does not appear to stabilise over time (\cref{fig:Oscillations}).
	
	\begin{figure}[h!]
		\centering
		\includegraphics[width=\textwidth]{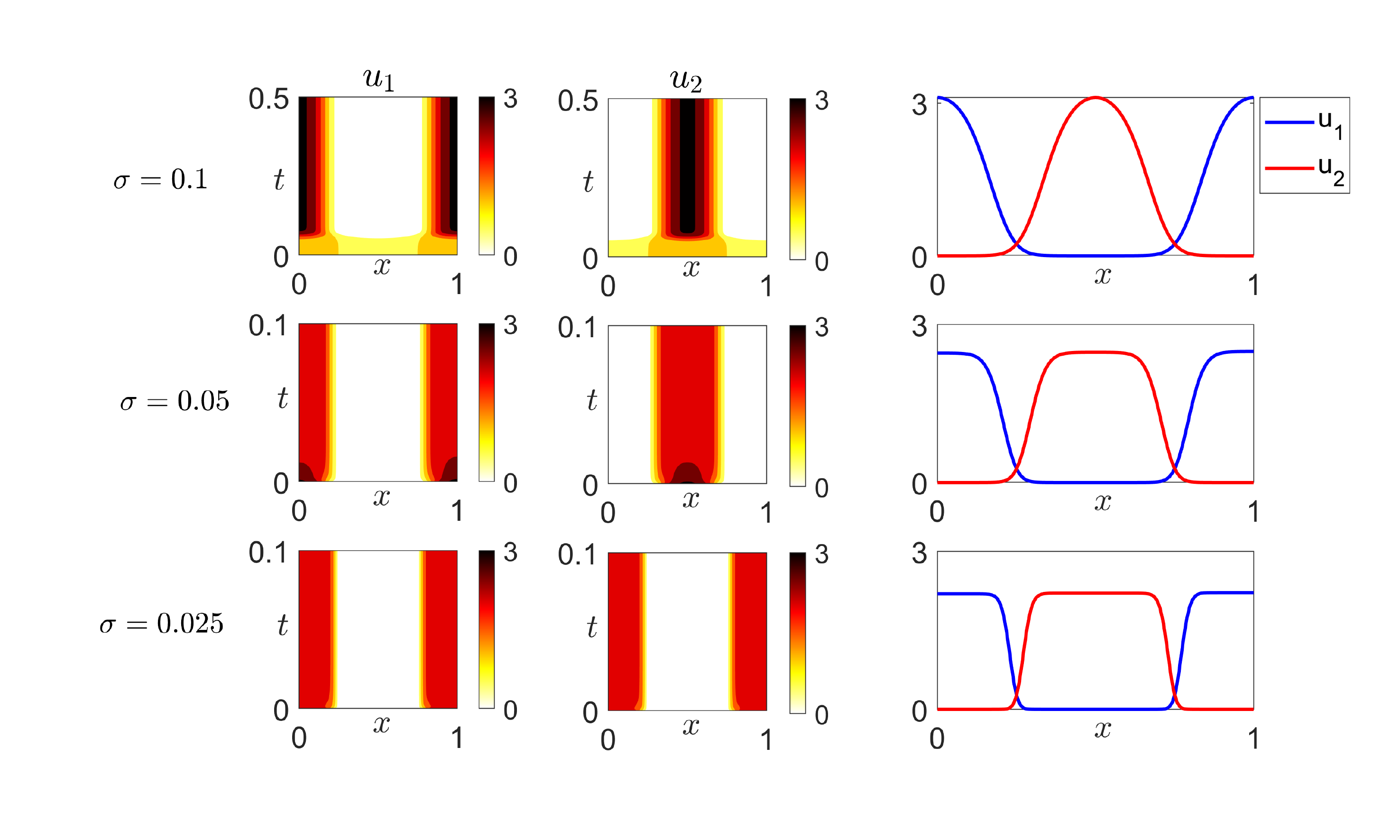}
		\caption{Spatio temporal evolution of the numerical solution of \eqref{eq:model1} with $ K=K_a $ defined in \Cref{eq:smoothk}, for different values of the standard deviation $ \sigma $. On the right column: spatial profile of the numerical stationary solution. The parameter values are: $ D_1=D_2=1 $, $ h_{11}=h_{22}=0 $, $ h_{12}=h_{21}=-2 $. For $ \sigma=0.1 $, $ a=3.225$; for $ \sigma=0.05 $, $ a=10.664$; for $ \sigma=0.025 $, $ a=41.01$.}
		\label{fig:StationarySolutions}
	\end{figure}
	
	\begin{figure}[h!]
		\centering
		\includegraphics[width=\textwidth]{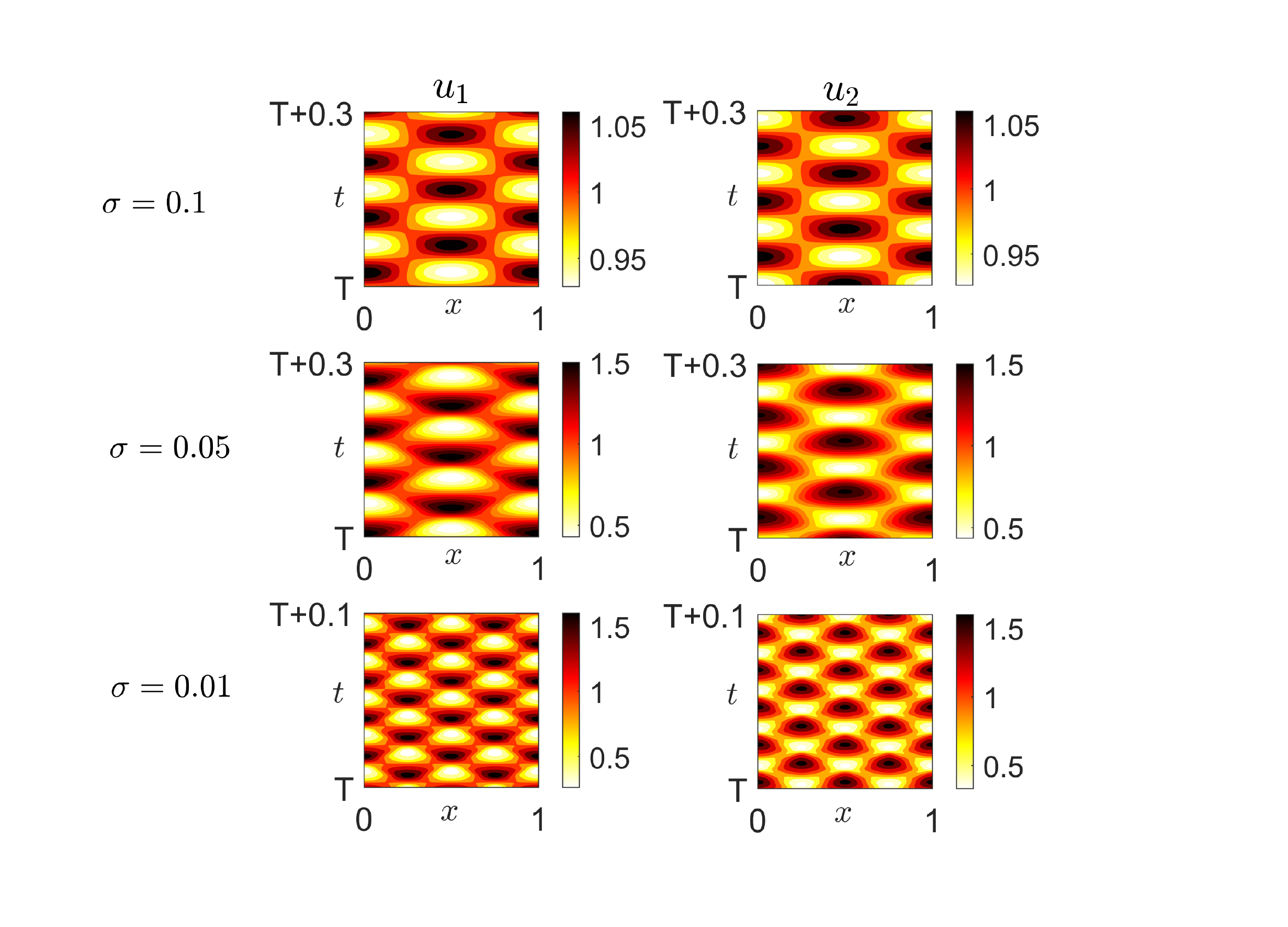}
		\caption{Spatio temporal evolution of the numerical solution of \eqref{eq:model1} with $ K $ defined in \Cref{eq:smoothk}, for different values of $ \sigma $. The parameter values are: $ D_1=D_2=1 $, $ h_{11}=h_{22}=h_{21}=1.5 $, $ h_{12}=-1 $. For $ \sigma =0.1 $, $ a=3.1 $; for $ \sigma=0.05 $, $ a=10.5 $; for $ \sigma=0.01 $, $ a=250 $.}
		\label{fig:Oscillations}
	\end{figure}
	
	Comparing the numerical solutions obtained with the von Mises kernel \eqref{eq:smoothk} and top-hat kernel \eqref{eq:nonsmoothk} for different values of $\sigma$, we see a good numerical agreement between numerical steady-state solutions (\cref{fig:ComparisonSmoothTopHat}). 	Hence, numerically, either choice is possible. 
	
	\begin{figure}[h]
		\centering
		\includegraphics[width=\textwidth]{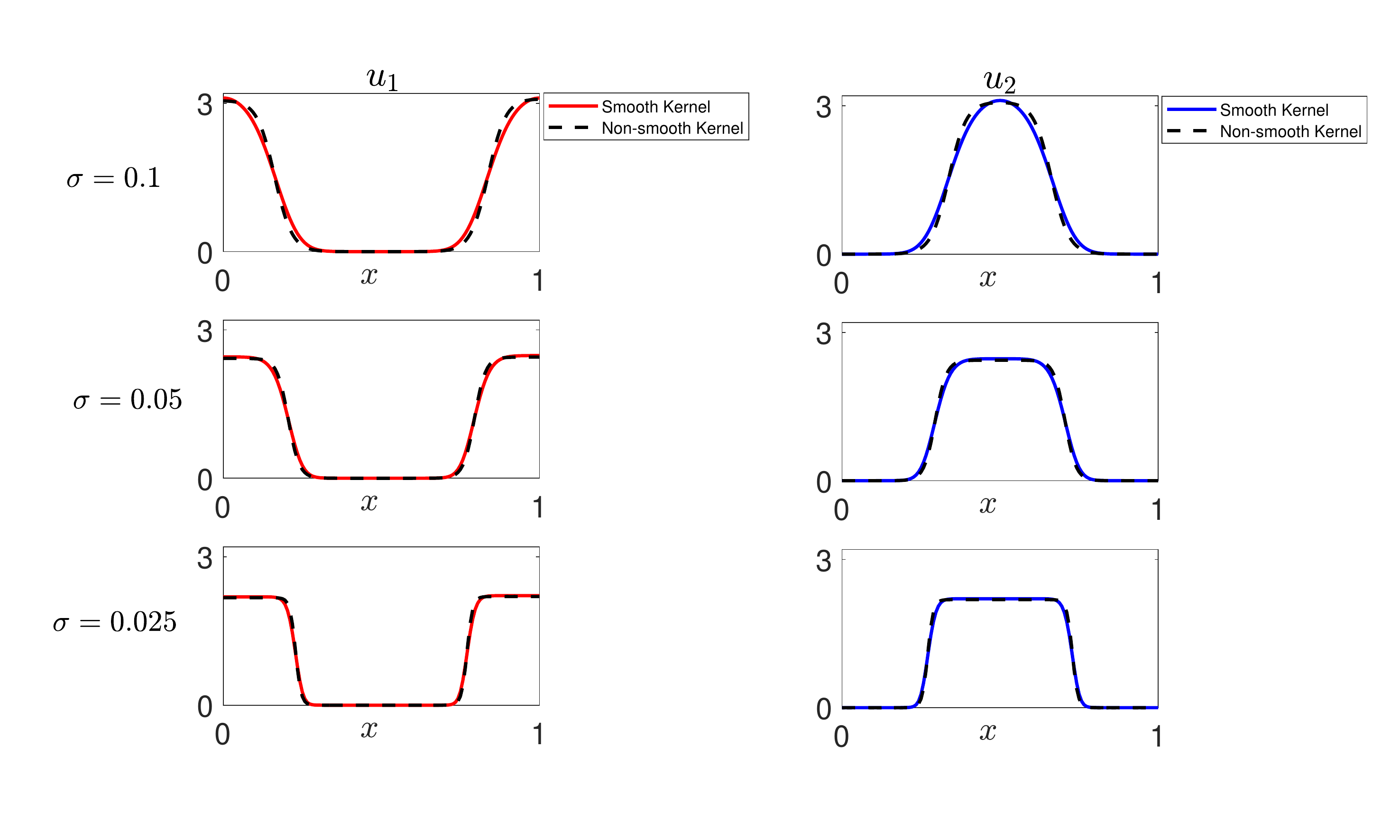}
		\caption{Comparison between the spatial profiles of the stationary solutions obtained with the smooth kernel K \eqref{eq:smoothk} and the non-smooth kernel $ K_\gamma $ \eqref{eq:nonsmoothk}, for different values of the standard deviation $ \sigma $. The parameter values are: $ D_1=D_2=1 $, $ h_{11}=h_{22}=0 $, $ h_{12}=h_{21}=-2 $. For $ \sigma=0.1 $, $ a=3.225 $ and $ \gamma=0.1732 $. For $ \sigma=0.05 $, $ a=10.664 $ and $ \gamma=0.0866 $. For $ \sigma=0.025 $, $ a=41.01 $ and $ \gamma=0.0433 $.}
		\label{fig:ComparisonSmoothTopHat}
	\end{figure}

	\section{Discussion}
	\label{sec:disc}

	The development of our model (\Cref{eq:msagg}) has been driven by the need to include non-local spatial terms into realistic models for organism interactions. However, when developing a new modelling framework, it is always a good idea to show that the model is well defined and biologically sensible, as we do here. In particular, it is important to identify the mathematical conditions that are needed to prove existence and uniqueness of solutions. In our case, for example, we find that the smoothness of the averaging kernel is essential to prove existence of classical solutions for the PDE model. %in dimensions two and higher.
	This implies that our favorite choice, the indicator function on a ball of radius $R$, used by \cite{pottslewis2019}, is not included in the  existence results. This is not a large restriction for the biology, since the indicator function can always be mollified (smoothed out) to obtain a regular kernel. However, it opens an interesting mathematical question to try to understand what goes wrong when the averaging kernel has jumps. In our case we cannot find a uniform $L^\infty$ estimate for convolution with $\nabla K$, which is an observation, but not an explanation of this limitation. In numerical simulations, we compare smooth and non-smooth averaging kernels and we see no appreciable difference. The difference is certainly much smaller than can ever be expected from errors that arise through empirical measurements of species distributions.  
	
	In our theory we consider a periodic domain, represented through the $n$-torus $\mathbb{T}$. Other domains with other boundary conditions can be studied with minimal modifications. The boundary conditions were essential to establish \cref{lemmaTaylor} about the regularity of the heat equation semigroup on $\mathbb{T}$. Similar regularity results are known for other boundary conditions \cite{LSU,Lieberman96}, and in those cases our method applies directly. 
	
	Non-local models for one or two species have been extensively studied before (see for example \cite{carrillo2018zoology,Eftimienonlocal} and the references that were mentioned in the Introduction). Our emphasis here is on a multiple species situation.  This system was originally introduced in \cite{pottslewis2019}, in a slightly modified form, for the purposes of understanding the effect of between-population movements on the spatial structure of ecosystems, something generally ignored in species distribution modelling \cite{dormann2018biotic}.  Understanding the spatial distribution of species has been named as one of the top five research fronts in ecology \cite{renner2013equivalence}, so the model presented here has potential for giving insights into various important problems in biology where biotic interactions affect movement.  These include, but are not limited to, the emergence of home range patterns \cite{borger2008there}, the geometry of selfish herds \cite{hamilton1971geometry}, the landscape of fear \cite{laundre2010landscape}, and biological invasions \cite{lewisetal2016}.
	
	The study of \cite{pottslewis2019}  focused on pattern formation via the tools of linear stability, numerical bifurcation, and energy functional analysis.  
	%A top-hat averaging kernel was used (Equation \ref{eq:nonsmoothk}) for the function $K$, but we were unable to obtain regularity results in this case so instead focused on twice-differentiable kernels.  However, it is easy to approximate the top-hat kernel with a twice-differentiable kernel that is arbitrarily-closes.  Furthermore, our numerical investigations suggest that very similar patterns can emerge from using a top-hat kernel as compared to an example smooth kernel, the von Mises distribution, which has a rather different shape.
	This study showed that the linear stability problem became ill-posed in the `local limit', i.e. as $K$ tends towards a Dirac delta function so that advection becomes non-local.  Analogously, here we show that solutions exist for smooth $K$, but depend upon $\|\nabla K\|_\infty$  being finite, so will also break down if $K$ is a Dirac delta function.  This highlights the importance of non-locality in our advection term.  Indeed, numerical simulations (e.g. \cref{fig:StationarySolutions}) suggest that, as $K$ narrows (i.e. its standard deviation decreases), the maximum gradient of any non-trivial stable steady state increases.  We conjecture that failure to include non-locality in the advection term (equivalently, setting $K$ to be a Dirac delta function) will lead to gradient blow-up.
	
	Our results, together with those of \cite{pottslewis2019}, suggest a rich variety of pattern formation properties in non-local multi-species advection-diffusion models.  Here, specifically, we see two new features related to pattern formation.  The first is the appearance of oscillatory solutions in two-species models, enabled by the inclusion of self-attractive terms.  Second, we see that changing the width of spatial averaging, given by $\sigma$, can have a qualitative effect on the patterns that emerge (\cref{fig:Oscillations}).  We have only scratched the surface here, in order to introduce our numerical method, the main purpose of this work being to establish existence of solutions.  Nonetheless, the ability to link underlying processes with emergent patterns is a principal question in biology \cite{hillenpainter2013,CosnerCantrellbook,potts2021stable}, and the evident rich pattern formation properties of these models suggest this will be a formidable task for future work, building on the increasing literature in this area \cite{pottslewis2016a,buttenschon2020non,carrillo2019aggregation}.

	%A quick numerical analysis, as presented here, shows a rich menu of spatio-temporal pattern formation. The analysis of these patterns is still beginning \cite{pottslewis2016a,buttenschon2020non,carrillo2019aggregation} and it offers an interesting direction of future exploration.  

%	\section*{Acknowledgements}
%	\label{sec:ack}
	
%	VG and JRP acknowledge the support of an Engineering and Physical Sciences Research Council (EPSRC) grant EP/V002988/1 awarded to JRP. VG also acknowledges support from GNFM-INdAM. TH is grateful for support from the Natural Science and Engineering Council of Canada Discovery Grant RGPIN-2017-04158. MAL gratefully acknowledges support from the NSERC Discovery and the Canada Research Chair programs.

	\bibliographystyle{siamplain}
	\bibliography{eup_refs}

\end{document}